\documentclass[12pt]{amsart}

\pdfoutput=1

\usepackage{latexsym}
\usepackage[centertags]{amsmath}
\usepackage{amsfonts}
\usepackage{amssymb}
\usepackage{amsthm}
\usepackage{newlfont}
\usepackage{graphics}
\usepackage{color}

\usepackage[usenames,dvipsnames]{xcolor}

\usepackage[demo]{graphicx} 
\usepackage{wrapfig,floatrow}
\usepackage{lipsum} 
\usepackage[font=small,labelfont=bf,margin=5mm]{caption}

\RequirePackage{xparse, graphicx, caption, picins}
\DeclareDocumentCommand \addpic{O{0.4\textwidth} m g}{\parpic[r]{%
\begin{minipage}{#1}
    \includegraphics[width=\textwidth]{#2}%
    \IfNoValueTF{#3}{}{\captionof{figure}{\footnotesize #3}}
\end{minipage}
}}

\newtheorem{theo}{Theorem}[section]

\newtheorem{exam}[theo]{Example}
\newtheorem{lem} [theo]{Lemma}

\theoremstyle{remark}

\newtheorem{algo}[theo]{Algorithm}

\def\SW{\texttt{SW}}

\def \TAU {{\cal T}}
\def \CD {{\cal D}}

\def \MK {{\mathbf{k}}}

\def \RA {\!\!\rightarrow\!\!}

\def\om {\omega}

\def\oD{\overline{D}}

\begin{document}
\vskip -.5in

\title{On the Sweep Map for $\vec{k}$-Dyck Paths}

\author{Guoce Xin$^1$ and Yingrui Zhang$^2$}

\address{ $^{1,2}$School of Mathematical Sciences, Capital Normal University,
 Beijing 100048, PR China}
\email{$^1$\texttt{guoce.xin@gmail.com}\ \  \& $^2$\texttt{zyrzuhe@126.com}}

\date{November 15, 2018} 

\begin{abstract}
Garsia and Xin gave a linear algorithm for inverting the sweep map for Fuss rational Dyck paths in $D_{m,n}$ where $m=kn\pm 1$. They introduced an intermediate
family $\mathcal{T}_n^k$ of certain standard Young tableau. Then inverting the sweep map is done by a simple walking algorithm on a $T\in \mathcal{T}_n^k$. We find their idea
naturally extends for $\mathbf{k}^\pm$-Dyck paths, and also for $\mathbf{k}$-Dyck paths (reducing to $k$-Dyck paths for the equal parameter case). The intermediate object becomes a similar type of tableau in $\mathcal{T}_\mathbf{k}$ of different column lengths.
This approach is independent of the Thomas-Williams algorithm for inverting the general modular sweep map.

\bigskip
\noindent
\begin{small}
 \emph{Mathematic subject classification}: Primary 05A19; Secondary 05E99.
\end{small}

\medskip
\noindent
\begin{small}
\end{small}


\end{abstract}

\maketitle

\section{Introduction\label{s-introduction}}

\subsection{The Sweep map}
The sweep map is a mysterious simple sorting algorithm that is also invertible.
The best way to introduce the sweep map is by using rational Dyck paths,
because it already raises complicated enough problems and it has natural generalizations.
We will use three models, each having their own advantages.

Model 1: Classical path model. For a
pair of positive integers $(m,n)$, the so called \emph{rational $(m,n)$-Dyck paths} are paths proceed by North and East unit steps from $(0,0)$ to
$(m,n)$ remaining always  above the main diagonal (of slope $n/m$). A $(kn,n)$-Dyck path is also called a $k$-Dyck path of length $n$.
Each vertex is assigned a rank as follows. We start by assigning $0$ to $(0,0)$. This done we add an $m$ as we go  North and subtract an $n$ as we go East. Figure \ref{fig:Dyck-Path-Example} illustrates an example of an $(12,4)$-Dyck path, or a $3$-Dyck path of length $4$.

\begin{figure}[!ht]
  $$\hskip 1.2in \oD= \hskip -1.4in\vcenter{ \includegraphics[height= 1.2 in]{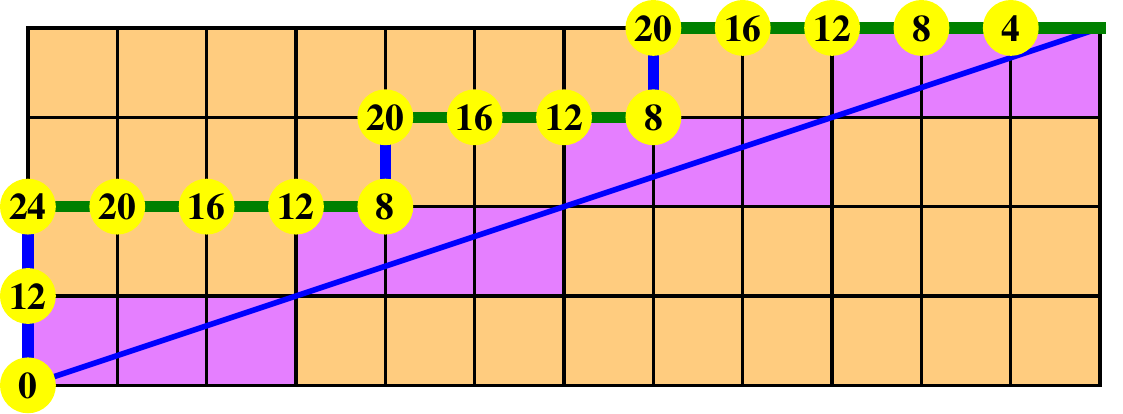}}
  $$
  \caption{An example of $(12,4)$-Dyck path.\label{fig:Dyck-Path-Example}}
\end{figure}
\begin{figure}[!ht]
$$
\hskip 1.15in D= \hskip -1.4in \vcenter{ \includegraphics[height=1.2 in]{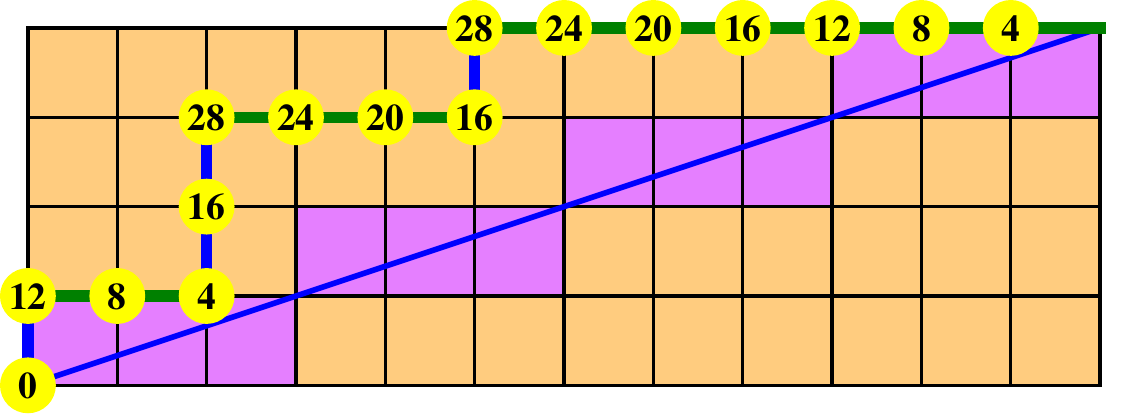}}
$$
\caption{ The sweep map image of the Dyck path $\oD$ in Figure \ref{fig:Dyck-Path-Example}.\label{fig:sweep-image}}
\end{figure}
The \emph{sweep} map is a bijection of the family
  ${\cal D}_{m,n}$ of $(m,n)$-Dyck paths onto itself. 
The construction of the sweep map is deceptively simple. Geometrically in model 1,
we sweep a path $\oD\in \CD_{m,n}$ by letting lines of slope $n/m+ \epsilon $, where $\epsilon>0$ is sufficiently small,  move from right to left,
and draw a North step when  we sweep the South end of a North step of $\oD$ and draw an East step when we sweep the West end of an East step of $\oD$.
The resulting path will be denoted by $D=\Phi(\oD)$ and can be shown to be in $\CD_{m,n}$. For instance, Figure \ref{fig:sweep-image} illustrates
the sweep map image of the Dyck path $\oD$ in our running example.

Model 2: Word model. The SW-word $\SW(\oD)=\sigma_1\sigma_2\cdots \sigma_{m+n}$ is a natural encoding of $\oD$, where $\sigma_i$ is either an $S$ or a $W$ depending on whether the $i$-th vertex of $\oD$ is a South end
(of a North step) or a West end (of an East step).  The rank is then associated to each letter of $\SW(\oD)$ by assigning  $r_1=0$ to the first letter $\sigma_1=S$ and for $1\le i\le m+n-1$, recursively
assigning $r_{i+1}$ to be either $r_i+m$ if the $i$-th letter $\sigma_i=S$, or $r_i-n$ if otherwise $\sigma_i=W$. We can then form the two line array $\left( \SW(\oD) \atop r(\oD)\right)$.
For instance for the path $\oD$ in Figure \ref{fig:Dyck-Path-Example} this gives
\begin{align}\label{e-II.1}
  \left(\SW(\oD) \atop r(\oD)\right)=
 \left (\begin{array}
   {cccccccccccccccc}
   S \!\!\! & S \!\!\! & W \!\!\! & W \!\!\! & W \!\!\! & W \!\!\! & S \!\!\! & W \!\!\! & W \!\!\! & W \!\!\! & S \!\!\! & W \!\!\! & W \!\!\! & W \!\!\! & W \!\!\! & W \\
   0 \!\!\! &  12 \!\!\! &  24 \!\!\! &  20 \!\!\! &  16 \!\!\! &  12 \!\!\! &  8 \!\!\! &  20 \!\!\! &  16 \!\!\! &  12 \!\!\! &  8 \!\!\! &  20 \!\!\! &  16 \!\!\! &  12 \!\!\! &  8 \!\!\! &  4
 \end{array} \right).
\end{align}
Now sort the columns of $(\ref{e-II.1})$ according to the second row, and let the right one comes first for equal ranks.
Then the top row is just the SW-word $\SW(D)$ of the sweep map image of $\oD$.
Our running example gives
\begin{align}
\left(\SW(D) \atop r(\overline{D})_{\texttt{sorted}}\right)= \left( \begin{array}{cccccccccccccccc}
S \!\!\! & W \!\!\! & W \!\!\! & S \!\!\! & S \!\!\! & W \!\!\! & W \!\!\! & W \!\!\! & S \!\!\! & W \!\!\! & W \!\!\! & W \!\!\! & W \!\!\! & W \!\!\! & W \!\!\! & W \\
0 \!\!\! & 4 \!\!\! & 8 \!\!\! & 8 \!\!\! & 8 \!\!\! & 12 \!\!\! & 12 \!\!\! & 12 \!\!\! & 12 \!\!\! & 16 \!\!\! & 16 \!\!\! & 16 \!\!\! & 20 \!\!\! & 20 \!\!\! & 20 \!\!\! & 24
\end{array} \right).
\end{align}

Model 3: Visual path model. We can rotate and stretch the picture in model $1$ so that the diagonal line becomes the horizontal axis. Then rational $(m,n)$-Dyck paths
are paths from $(0,0)$ to $(m+n,0)$, with up steps $S=(1,m)$ and down steps $W=(1,-n)$, that stay above the horizontal axis. It is clear that the ranks
are just the levels (or $y$-coordinates). The visualization of the ranks in this model allows us to have better understanding of many results. See \cite{dinv-area,Rational-Invert}.
See \cite{dinv-def,qt-Catalan,Gorsky-Mazin,Hag-book08,Loehr-higher-qtCatalan} for further references and related results. 

The invertibility problem is to reconstruct $\oD$ from the sole knowledge of
$\SW(D)$.

The  sweep map is an active subject in the recent decades. Variations and extensions have been found, and some classical bijections turn out to be
the disguised version of the sweep map. See \cite{sweepmap} for detailed information and references. One major problem in this subject
is the invertibility of the sweep map. The bijectivity has been shown in a variety of special  cases including
the Fuss case  when $m=kn\pm1$ which is proved in \cite{Loehr-higher-qtCatalan},\cite{Gorsky-Mazin2}and\cite{Fuss-case}.
The invertibility of the sweep map, even for rational Dyck paths, remained open for over ten years.
Surprisingly, a general result proving the invertibility of a class of sweep maps that were listed in \cite{sweepmap}, was recently given by
Thomas-Williams in \cite{Nathan}. Based on the idea of Thomas and Williams, Garsia and Xin \cite{Rational-Invert} gave a geometric construction for inverting the sweep on $(m,n)$-rational Dyck paths.
These algorithms are nice iteration algorithms, but are not linear: the number of iterations is measured by the sum of the ranks of $\oD$.

By using a completely different approach, Garsia and Xin find a $O(m+n)$ algorithm for inverting sweep map on $(m,n)$-Dyck paths in the Fuss case $m=kn \pm 1$.
This raises the following problem.
\def\k{\mathbf{k}}

\noindent
\textbf{Problem:} Is there a linear algorithm to invert the sweep map, at least for a more general class of paths?

We find positive answer for $\k^+$, $\k^-$, and $\k$-Dyck paths by extending Garsia-Xin's idea.

\subsection{The notation of general Dyck paths}
We start by introducing general Dyck paths using model 3.

\emph{General Dyck paths} are two dimensional lattice paths from $(0,0)$ to $(m+n,0)$ that never go below the horizonal axis.
We use vectors $\mathbf{u}=(u_1,\dots, u_n)$ and $\mathbf{d}=(d_1,\dots, d_m)$ to specify the up steps and down steps,
so that $\CD_{\mathbf{u},-\mathbf{d}}$ is the set of general Dyck paths with up steps $(1,u_i), \ 1\le i\le n$ from left to right,
and down steps $(1,-d_j), \ 1\le j \le m$ from left to right. Clearly the total length of up steps $|\mathbf{u}|=u_1+u_2+\cdots +u_n$
is equal to the total length of down steps $|\mathbf{d}|=d_1+d_2+\cdots +d_m$.
When each $d_i=1$, we use the short hand notation $\CD_{\mathbf{u}}=\CD_{\mathbf{u},-\mathbf{1}_{|\mathbf{u}|}}$, where $\mathbf{1}_a$ consists of $a$ $1$'s.
Here we usually restrict $u_i$ and $d_j$ to be positive integers, but sometimes allowing rational numbers is convenient.

A general Dyck path $\oD$ may be encoded as $\oD=(a_1,a_2,\dots, a_{m+n})$ with each entry either $u_i$ or $-d_j$. The rank sequence $r(\oD)=(0=r_1,r_2,\dots, r_{m+n})$ of $\oD$ is defined as the partial sums
$r_i=a_1+a_2+\cdots +a_{i-1}\ge 0$, called starting rank (or level) of the $i$-th step. Geometrically, $r_i$ is just the level or $y$-coordinate of the starting point of the $i$-th step.
The SW-word of $\oD$ is $\SW(\oD)=\sigma_1\sigma_2\cdots \sigma_{m+n}$ where $\sigma_i=S^{a_i}$ if $a_i>0$ and $\sigma_i=W^{-a_i}$ if $a_i<0$ (with $W^1=W$).
The sweep map $D$  of $\oD$ is obtained by reading its steps by their starting levels from bottom to top, and from right to left in the same level.
This corresponds to sweeping the starting points of the steps from bottom to top using a line of slope $\epsilon$ for sufficiently small $\epsilon>0$.

\begin{figure}[!ht]
  $$
 \hskip .005in \oD= \hskip -1.95in\vcenter{ \includegraphics[height= 1.06 in]{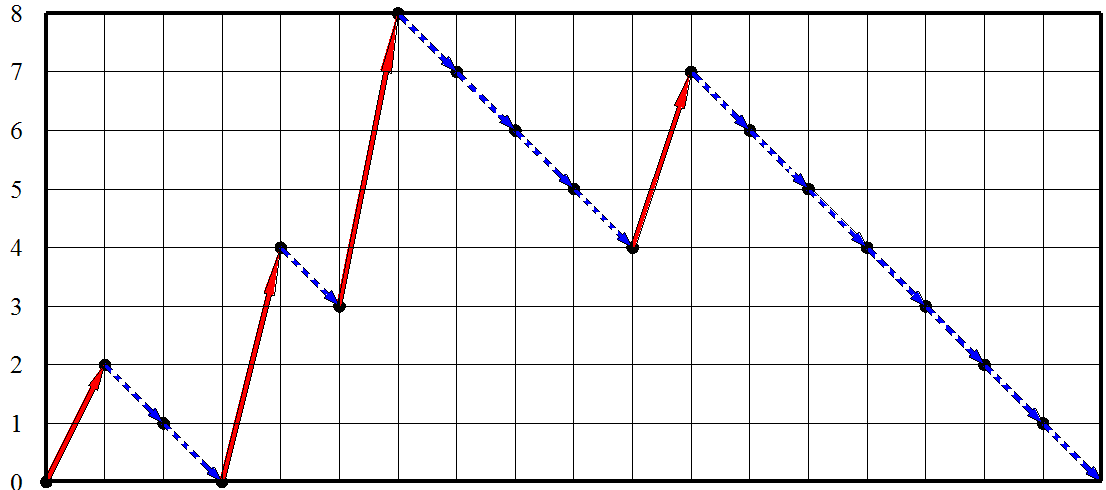}}
\hskip -2.2in \qquad \Longrightarrow\qquad\hskip -0.3inD= \hskip -1.95in \vcenter{ \includegraphics[height=1.06 in]{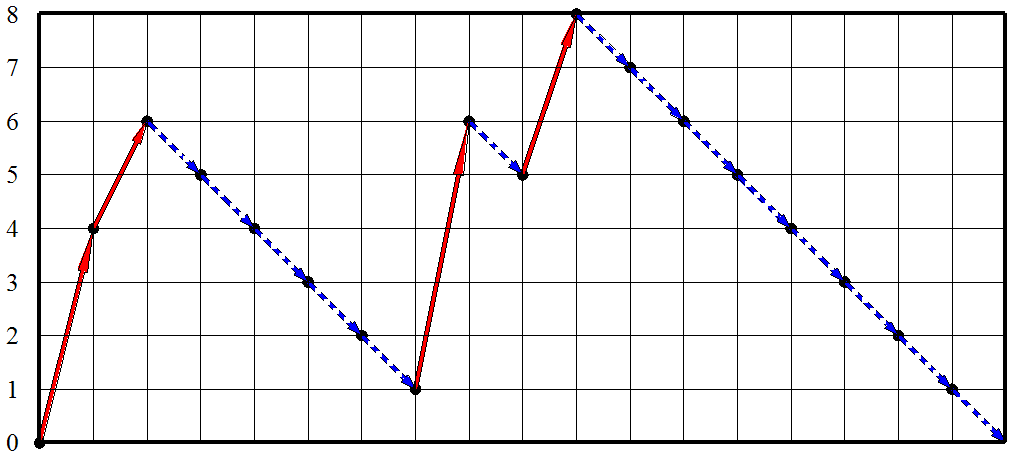}}
$$
\caption{An example of a $\k$-Dyck path and its sweep map image.
\label{fig:model3-Example}}
\end{figure}

Figure \ref{fig:model3-Example} illustrates an example of $\k$-Dyck path $\overline{D}$ given by
$$ \overline{D}=(2,-1,-1,4,-1,5,-1,-1,-1,-1,3,-1,-1,-1,-1,-1,-1,-1),$$
where $\k=(2,4,5,3)$. The SW-word of $\overline{D}$ and the rank sequence are given by
$$ \left(\SW(\overline{D}) \atop r(\overline{D})   \right)=
\left(\begin{array}{cccccccccccccccccc}
                                        S^2  \!\!\! &  W  \!\!\! &  W  \!\!\! &  S^4  \!\!\! &  W  \!\!\! &  S^5  \!\!\! &  W  \!\!\! &  W  \!\!\! &  W  \!\!\! &  W  \!\!\! &  S^3  \!\!\! &  W  \!\!\! &  W  \!\!\! &  W  \!\!\! &  W  \!\!\! &  W  \!\!\! &  W  \!\!\! &  W \\
                                        0    \!\!\! &  2  \!\!\! &  1  \!\!\! &  0    \!\!\! &  4  \!\!\! &  3    \!\!\! &  8  \!\!\! &  7  \!\!\! &  6  \!\!\! &  5  \!\!\! &  4    \!\!\! &  7  \!\!\! &  6  \!\!\! &  5  \!\!\! &  4  \!\!\! &  3  \!\!\! &  2  \!\!\! &  1 \\
                                      \end{array}\right). $$
The sweep map image is
$$D=(4,2,-1,-1,-1,-1,-1, 5,-1,3,-1,-1,-1,-1,-1,-1,-1,-1),$$
with SW-word $\SW(D)=S^4 S^2 W W W W W S^5 W S^3 W W W W W W W W$.

\begin{exam}
  \begin{enumerate}
    \item $\CD_{\mathbf{1}_n}$ is the set of classical Dyck paths in the $n\times n$ square (rotated version).
    \item $\CD_{m\mathbf{1}_n, -n\mathbf{1}_m}$ is just $\CD_{m,n}$, the set of rational $(m,n)$-Dyck paths.

    \item $\CD_{k\mathbf{1}_n}$ is just $\CD_{kn,n}$, the set of  $k$-Dyck paths of length $n$. (Their rank sequences differed by a factor $n$).
  \end{enumerate}
\end{exam}

In what follows, $\k=\vec{k}=(k_1,k_2,\dots, k_n)$ is always a vector of $n$ positive integers.
We will focus on $\mathbf{k}$-Dyck paths in $\CD_{\k}$, i.e., general Dyck paths whose up steps are of lengths $k_1,\dots, k_n$ from left to right, and whose down steps
are all of length $1$. We will also consider $\mathbf{k}^+$-Dyck paths in $\CD_{\k^+}$ and $\mathbf{k}^-$-Dyck paths in $\CD_{\k^-}$, where
$\k^{\pm}=\k\pm \frac{1}{n} \mathbf{1}_n$. These are natural extension of the Fuss case rational Dyck paths: $k$-Dyck paths are
 just $k\mathbf{1}_{n}$-Dyck paths; Fuss case $(kn\pm 1,n)$-Dyck paths are easily seen to be equivalent to $\k^\pm$-Dyck paths.

The sweep map of a $\k$-Dyck path is usually a $\k'$-Dyck path where $\k'$ is obtained from $\k$ by permuting its entries. Denote by $\mathcal{K}$ the set of
all such $\k'$ and by $\CD_{\mathcal{K}}$ the union of $\CD_{\k'}$ for all such $\k'$. We also use similar notation for $\mathcal{K}^\pm$ and
$\CD_{\mathcal{K}^\pm}$.

\begin{theo}
The sweep maps define bijections from $\CD_{\mathcal{K}^+}$, $\CD_{\mathcal{K}^-}$, and $\CD_{\mathcal{K}}$ to themselves.
\end{theo}

It is known but nontrivial that the sweep map takes a Dyck path to another Dyck path (see, e.g., \cite{Rational-Invert} for a proof), so the proof of the theorem boils down to construct
the inverse image $\overline{D}$ from a given Dyck path $D$.

The three sets $\CD_{\k}$, $\CD_{\k^+}$, and $\CD_{\k^-}$ are closely related as follows. For $D\in \CD_{\k}$, let $\SW(D^+)$ be obtained from $\SW(D)$ by adding
a $W$ at the end and changing every $S^a$ to $S^{a+1/n}$. Similarly let $\SW(D^-)$ be obtained from $\SW(D)$ by removing
the final letter $W$ and changing every $S^a$ to $S^{a-1/n}$. It is clear that the map $D\mapsto D^+$ gives a bijection from
$\CD_{\k}$ to $\CD_{\k^+}$. However, the map $D\mapsto D^-$ is a little different: it is a bijection from $\CD_{\k}^\circ$ to $\CD_{\k^+}$,
where $\CD_{\k}^\circ$ consists of paths $D\in \CD_{\k}$ whose rank sequence $r(D)$ has only one $0$ at $r_1$.

Figure $\ref{fig:idea}$  illustrates the idea: Though the paths ${D^+}\in\CD_{\k^+}$ and $D\in\CD_{\k}$ have almost identical pictures, their sweep map inverse images $\overline{D^+}\in\CD_{\k'^+}$ and $\oD\in\CD_{\k''}$ may be very different, due to the different sweep order.
Therefore, their inverting algorithma are also very different.

\begin{figure}[!ht]
  \begin{center}
 \includegraphics[height= 1.3 in]{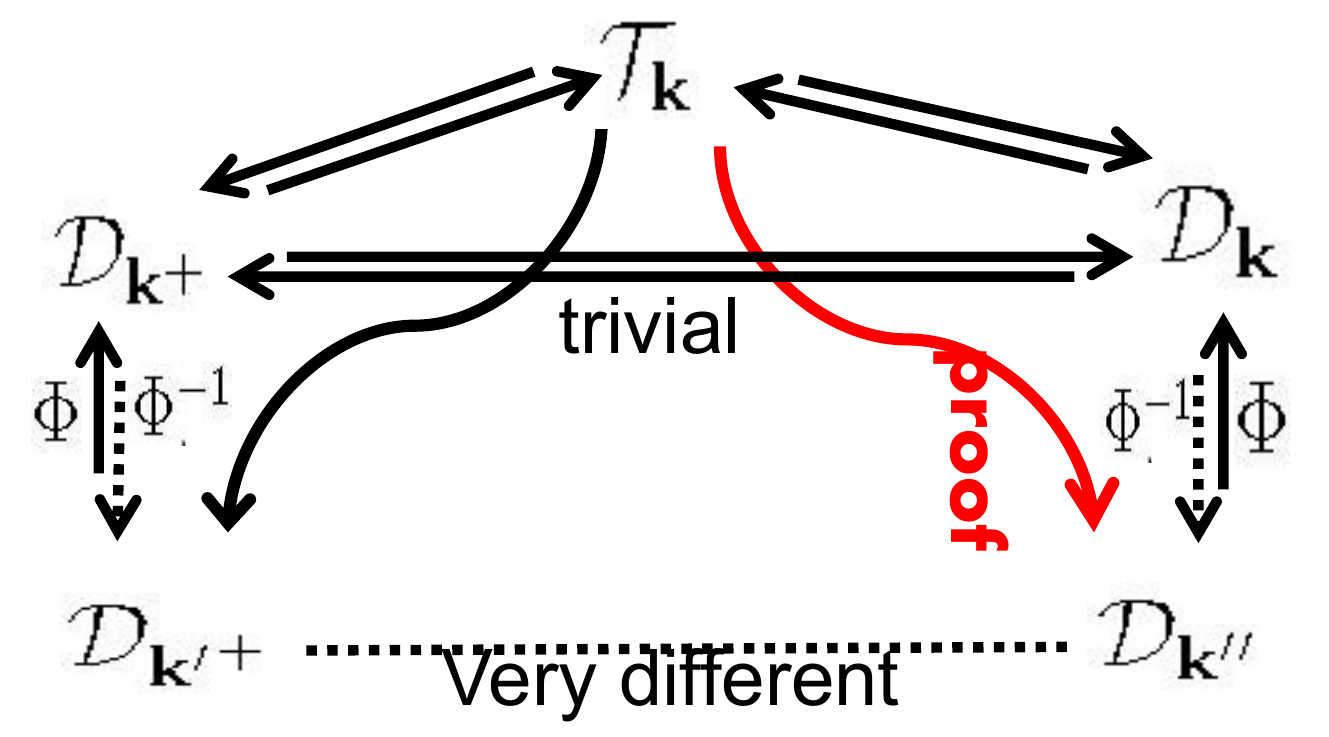}
  \end{center}
  \caption{The idea for inverting the sweep map: Solid curve means easy, and dotted curve means difficult. So to obtain $\overline{D}$, we need the help of $T$.
  \label{fig:idea}}
\end{figure}

\subsection{The Filling algorithm and the $\vec{\mathbf{k}}$ tableaux}
In the Fuss case when $m=kn+1$, a linear algorithm to invert the sweep map was discovered by Garsia and Xin in \cite{Fuss-case}. The algorithm relies on an intermediate family $\TAU_n^k$ of standard Young tableaux.
The family $\TAU_n^k$ consists of $n\times (k+1)$ arrays with entries $1,2,\ldots ,kn+n$, column increasing from top to bottom and row increasing from left to right, with the additional property that for any pair of entries $a<d$  with $d$ directly below $a$, the entries between $a$ and $d$ form a horizontal strip. That is, any pair of entries  $b,c$ with $a<b<c<d$ never appear in the same column. The standard Young tableau $T(D)$ constructed from the SW word of a path $D$  encodes so much information about $D$. This allows us to invert the sweep map in the simplest possible way.

We find Garsia-Xin's construction naturally extends for $\k^+$-Dyck paths. The intermediate family $\TAU_n^k$ becomes the family $\TAU_{\k}$,
where the only difference is that tableau $T\in \TAU_{\k}$ has $k_i+1$ entries in the $i$-th column. The Filling Algorithm \cite{Fuss-case} is adapted in our case as follows, where the major
change is the definition of \emph{active}.
\begin{algo}[Filling Algorithm]
\label{al-Filling Algorithm}

\noindent
Input: The SW-sequence $\SW(D)$ of a $\k$-Dyck path $D\in \CD_{\k}$.

\noindent
Output: A tableau $T=T(D)\in \TAU_{\k}$.

\begin{enumerate}
\item   Start by placing a $1$ in the top row and the first column.

\item  If the second letter in $\SW(D)$ is an $S^*$ we put a $2$ on the top of the second column.

\item   If the second letter in $\SW(D)$ is a $W$ we place $2$ below the $1$.

\item  At any  stage  the entry at the bottom of the $i$-th column but not in row $k_i+1$ will be called \emph{active}.

\item  Having placed $1,2,\cdots i-1$, we place $i$ immediately below the smallest  active entry if the $i^{th}$ letter in $\SW(D)$ is a $W$, otherwise we place $i$ at the top of the first empty column.

\item   We carry this out recursively until $1,2,\ldots ,n+|\k|$ have  all  been placed.
\end{enumerate}
\end{algo}

We will denote by $t(T)=(t_1,\dots, t_n)$ the top row entries of $T$ from left to right, and similarly by $b(T)=(b_1,\dots, b_n)$ the bottom entries.
Note that the former is always increasing, but the latter is not.
See Figure $\ref{fig:Filling Algorithm}$ for an example.
\begin{figure}[!ht]
 $$
 \hskip .15in T= \hskip -2.6in\vcenter{ \includegraphics[height= 1.5 in]{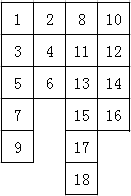}}
 \hskip -2.3in
 \begin{array}{l}
 SW(D)=S^4 S^2 W W W W W S^5 W S^3 W W W W W W W W\\
 \MK=(4,2,5,3)\\
 t(T)=(1, 2, 8, 10)\\
  b(T)=(9, 6, 18, 16)
 \end{array}
$$
 \caption{The path $D$ in Figure \ref{fig:model3-Example} and its filling tableau by the Filling algorithm.
 \label{fig:Filling Algorithm}}
\end{figure}

It is clear that the top row entries $t(T)$ uniquely determines $T$. Moreover, $\SW(D)$ can be recovered from
$T(D)$ by  placing  letters $S^{k_i}$ on the positions $t_i$ of $T(D)$ and letters $W$ in all the remaining $|\k|$ positions.
Indeed, we have the following characterization.

\begin{lem}
A sequence $(t_1,\dots, t_n)$ is the top row entries $t(T)$ for some $T\in \TAU_\k$ if and only if $t_i\le k_1+\cdots+k_{i-1}+i$ holds for all $i$.
\end{lem}

\begin{theo}
The Filling algorithm defines a bijection from $\CD_{\k} $ to $\TAU_{\k}$.
\end{theo}

\subsection{Walking algorithm for $\vec{\mathbf{k}}^\pm$-Dyck paths}

The walking algorithm for inverting the sweep map on $\CD_{kn\pm 1,n}$ naturally extends to that of  $\k^\pm$-Dyck paths.
To state our results,
we need to modify some notations.
For $T\in \TAU_{\k}$, let $T^+$ be obtained from $T$ by adding $n+|\k|+1$ below the entry $n+|\k|$. Let $\TAU_{\k}^+= \{ T^+: T\in \TAU_{\k}\}$.
The bottom entry of the $i$-th column of $T^+$ refers to the $(k_i+1)$-st entry for all $i$ including the column
with the added entry $n+|\k|+1$.
\begin{algo}[Walking Algorithm$^+$]
\label{al-Walking Algorithm^+}

\noindent
Input: A tableau $T^+=T(D^+)\in \TAU_{\k}^+$.

\noindent
Output: A permutation $\sigma(D^+)$ through walk  on $T^+$.

\begin{enumerate}
\item Write in bold  all the entries in $T^+$ that are by 1 more than a bottom row entry.
\item Go to 1 and write 1.
\item If you are in row  $1$ go down the column to the bottom. If the entry there is $r$, then go to $r+1$ and write  $r+1$.
\item If you are not in the first row go up the column one row. If the entry there is $r$ and is not bold write $r$.
\item If the entry there  is $r$ and bold go to $r-1$ and continue until
you run into a normal entry, then write it.
\end{enumerate}
\end{algo}
Once the permutation $\sigma(D^+)$ is obtained, the SW-word $\ SW\big(\overline{D^+}\big)$ of the inverse image
$\overline{D^+}$ is easy to construct: we write one letter at a time by placing above each entry of $\sigma(D^+)$ an $S^{{k_i^+}}$ if that entry is $t_i$ and a $W$ if that entry is not in row $1$.
This done we can simply draw $\overline{D^+}$ by reading the sequence of letters of $\ SW\big(\overline{D^+}\big)$. Note that
we will also use this construction for $\sigma(D^-)$ and $\sigma(D)$ in a similar way.

\begin{theo}For a Dyck path $D^+\in \CD_{\k^+}$, the permutation $\sigma(D^+)$ that rearranges the letters of $\SW(D^+)$ in the order that gives the $SW$ word of $\overline{D^+}$, is obtained by
the Walking Algorithm$^+$(Algorithm $\ref{al-Walking Algorithm^+}$).
\end{theo}

The proof is almost the same, so we only sketch the idea and point out the difference. Suppose $b_i=\min(b(T))$, i.e., the smallest bottom entry is in the $i$-th column. Then this column has no
bold faced entry. (In $\TAU^k_n$, this $i$ is always $1$.) Removing column $i$ results in a tableau ${T'}^+$, which still satisfies the $abcd$ condition. Then Walking Algorithm$^+$ applies to ${T'}^+$
to give $\sigma(D'^+)$ (regard the entries as contiguous just like the Fuss case). The similarity between $\sigma(D^+)$ and $\sigma(D'^+)$ can be used to give the proof.
See Figure \ref{fig:Walking Algorithm^+} for an example.

\begin{figure}[!ht]
$$
\hskip 1.0in T^+= \hskip -2.65in \vcenter{ \includegraphics[height=1.5 in]{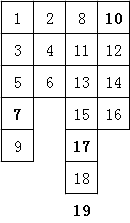}}
\hskip -1.4in {T'}^+= \hskip -2.65in \vcenter{ \includegraphics[height=1.5 in]{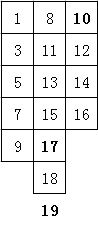}}
$$
$$
\hskip -0.3in\sigma(D^+)= \begin{array}{ccccccccccccccccccc}
\!\!\!1&\!\!10 &\!\!\!17 & \!\!\!15 & \!\!\!13 & \!\!\!11 & \!\!\!8 & \!\!\!19 & \!\!\!18 & \!\!\!16 & \!\!\!14&\!\!\!12 & \!\!\!9 &\!\!\! {\color{red}\mathbf{6}} & \!\!\!{\color{red}\mathbf{4}}& \!\!\!{\color{red}\mathbf{2}}& \!\!\!7& \!\!\!5& \!\!\!3\\
\end{array}
$$
$$
\hskip -0.3in\sigma({D'}^+)= \begin{array}{ccccccccccccccccccccc}
\!\!\!1&\!\!\!10 &\!\!\!17 &\!\!\!15 &\!\!\!13 &\!\!\!11 & \!\!\!8 & \!\!\!19 & \!\!\!18 & \!\!\!16 & \!\!\!14&\!\!\!12 & \!\!\!9& & & &7& \!\!\!5& \!\!\!3\\
\end{array}
$$
\caption{Walking Algorithm$^+$ applies to $T^+$ to give $\sigma(D^+)$, and Walking Algorithm$^+$ applies to $T'^+$ which is $T^+$ removing the column $2$ to give $\sigma(D'^+)$.\label{fig:Walking Algorithm^+}}
\end{figure}

\medskip
The situation for $\CD_{\k^-}$ is similar. Recall that $ D^-\in \CD_{\k^-}$ only when $D\in \CD_{\k}^\circ$, i.e., the rank sequence $r(D)$
has only one $0$ at $r_1=0$. It is not hard to see that the filling algorithm takes such $D$ to
$$\TAU_{\k}^- = \{ T\in \TAU_{\k}: t_i < k_1+\cdots +k_{i-1}+i \text{ for } i\ge 2\} ,$$
where $t_i$ denotes the top entry in the $i$-th column of $T$.

\begin{algo}[Walking Algorithm$^-$]
\label{al-Walking Algorithm^-}

\noindent
Input: A tableau $T^-=T(D^-)\in \TAU_{\k}^-$.

\noindent
Output: A permutation $\sigma(D^-)$ through walking on $T^-$.

\begin{enumerate}
  \item Write in bold  all the entries in $T^-$ that are by $1$ less than a bottom row entry.
\item Go to $1$ and write $1$.
\item If you are in  row  $1$ go down the column to the bottom. If the entry there is $r$  go to  $r-1$   and write  $r-1$.
\item If you are not in the first row go up the column one row. If the entry there is $r$ and is not bold write $r$.
\item If the entry there  is $r$ and bold go to $r+1$ and continue until
you run into a normal entry, then write it.

\end{enumerate}
\end{algo}

\begin{theo}For a Dyck path $D^-\in \CD_{\k^-}$, the permutation $\sigma(D^-)$ that rearranges the letters of $\SW(D^-)$ in the order that gives the $SW$ word of $\overline{D^-}$, is obtained by
the Walking Algorithm$^-$.
\end{theo}

The idea of the proof is similar to that of the $\k^+$-case. We only point out the difference. Suppose $b_i=\min(b(T))$, i.e., the smallest bottom entry is in the $i$-th column. (In $\TAU^k_n$, this $i$ is always $1$.)
Removing column $i$ results in a tableau ${T'}^-$, which still satisfies the $abcd$ condition. Then Walking Algorithm$^-$ applies to ${T'}^-$
to give $\sigma(D'^-)$ (regard the entries as contiguous just like the Fuss case). The similarity between $\sigma(D^-)$ and $\sigma(D'^-)$ can be used to give the proof. It is possible that column $i$ has a bold faced entry.
Then this bold faced entry must be the bottom entry $b_i$: i) the smallest bold faced entry $b_i-1$ is not in this column; ii) $b_i$ is bold faced only when $b_i=b_j-1$ for some $j$. This case also cause no trouble.
See Figure \ref{fig:Walking Algorithm^-} for an example.

\begin{figure}[!ht]
$$
\hskip 1.0in T^-= \hskip -2.65in \vcenter{ \includegraphics[height=1.5 in]{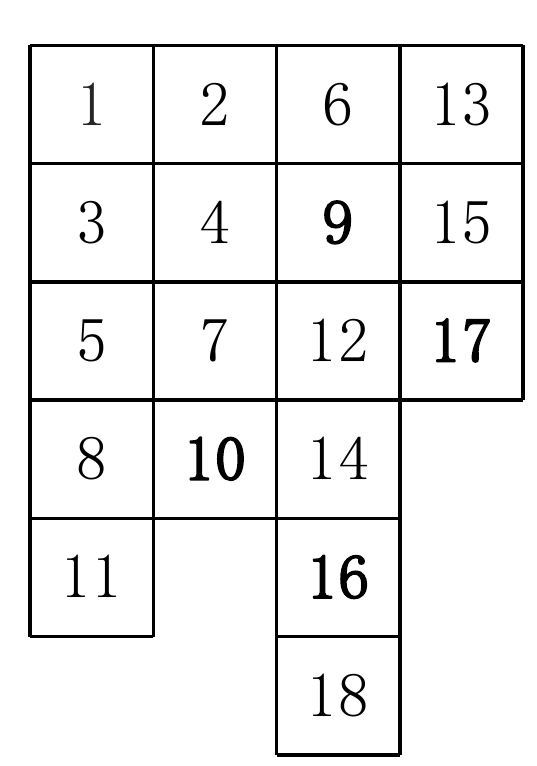}}
\hskip -1.4in T'^-= \hskip -2.65in \vcenter{ \includegraphics[height=1.5 in]{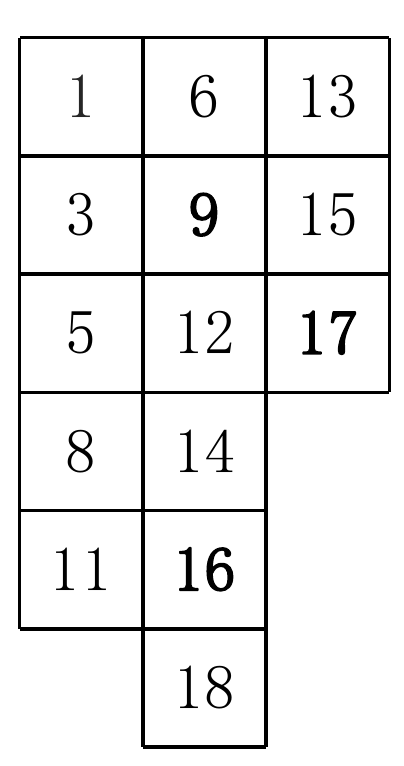}}
$$
$$
\hskip -0.3in\sigma(D^-)= \begin{array}{ccccccccccccccccc}
\!\!\!1&\!\!\!{\color{red}10} &\!\!\!{\color{red}7} & \!\!\!{\color{red}4} &\!\!\!{\color{red}2} & \!\!\!9 & \!\!\!6 & \!\!\!17 & \!\!\!15 & \!\!\!13 & \!\!\!16&\!\!\!14& \!\!\!12& \!\!\!11&\!\!\!8& \!\!\!5&\!\!\!3\\
\end{array}
$$
$$
\hskip -0.3in\sigma(D'^-)= \begin{array}{ccccccccccccccccc}
\!\!\!1 & & & &  & 9 & \!\!\!6 & \!\!\!17 & \!\!\!15 & \!\!\!13 & \!\!\!16&\!\!\!14& \!\!\!12& \!\!\!11&\!\!\!8& \!\!\!5&\!\!\!3\\
\end{array}
$$
\caption{Walking Algorithm$^-$ applies to $T^-$ to give $\sigma(D^-)$, and Walking Algorithm$^-$ applies to $T'^-$ which is $T^-$ removing the column $2$ to give $\sigma(D'^-)$.\label{fig:Walking Algorithm^-}}
\end{figure}

\subsection{The walking algorithm for $\vec{\mathbf{k}}$-Dyck paths}

After extending Garsia-Xin's idea for $\k^\pm$-Dyck paths, one might think the inverting algorithm will be similar for $\k$-Dyck paths. It turns out that the situation is quit different.

The new walking algorithm not only relies on the intermediate tableau $T=T(D)\in \TAU_{\k}$, but also on the \emph{rank tableau} $R(D)$ constructed from $T$.

It is convenient to call numbers in $T$ indices. We will use a ranking algorithm  to construct the rank tableau $R(D)$ of $T$.
For clarity, we start with the empty tableau of the shape of $T$, and successively \emph{assign} each index a rank.
By assigning index $A$ a rank $r$, we mean to fill $r$ into the box in $R(D)$ corresponding to index $A$ in $T$.

\begin{algo}[Ranking Algorithm]
\label{al-Ranking Algorithm}

\noindent
Input: A tableau $T=T(D )\in \TAU_{\k}$.

\noindent
Output: A rank tableau $R(D)$ of the same shape with $T$.

\begin{enumerate}
\item  Successively assign $0, 1, 2, ..., k_1$ to the first column indices of $T$ from top to bottom;

\item For $i$ from $2$ to $n$, if the top index of the $i$-th column is $A+1$, and the rank of index $A$ is $a$, then assign the index $A+1$ rank $a$. Moreover,
the ranks in the $i$-th column are successively $a,a+1,\dots, a+k_i$ from top to bottom.

%
%

\end{enumerate}
\end{algo}


\noindent
 See Figure \ref{fig:Ranking Algorithm} for an example of the Filling algorithm.
\begin{figure}[!ht]
$$
\hskip .8in T(D)=\hskip -2.55 in \vcenter{ \includegraphics[height=1.6in]{Tanytableaux.png} }
\hskip -2.5in \qquad \Longrightarrow \qquad
\hskip 0.05in
{R(D )=\hskip-2.5in \vcenter{\includegraphics[height=1.5in]{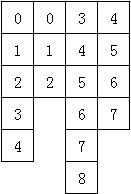} }}
$$
\caption{The tableau $T(D)$ in Figure \ref{fig:Filling Algorithm} and its ranking tableau.\label{fig:Ranking Algorithm}}
\end{figure}

Observe that the indices are distinct, but the ranks are not. The largest rank $r$ (entry) is the rank $r$  with the
largest index. For instance in 
Figure \ref{fig:Ranking Algorithm},
the rank $2$ entries have indices $5,6$, so the largest rank $2$ entry is the rank $2$ with index $6$, whose box is located in row $3$ column $2$. Similarly, the smallest rank $2$ entry is the rank $2$ with index $5$, whose box is located in row $3$ column $1$.

We find a way to obtain the permutation $\sigma(D)$ directly from $T(D)$ and
$R(D)$ as follows.
\begin{algo}[Walking Algorithm]
\label{al-Walking Algorithm}

\noindent
Input: The index-rank tableau $(T,R)$ with $T=T(D )\in \TAU_{\k}$ and $R=R(D)$.

\noindent
Output: A permutation $\sigma(D)$ through walking on $(T,R)$.

\begin{enumerate}
\item In $R(D)$, go to the largest rank $0$ entry. Mark this rank  and write down its index;

\item Repeat the following steps until no unmarked rank can be selected.
\begin{enumerate}
  \item If we are in row $1$, then go to the bottom row in the same column. If the rank there is $r$ then go to the largest unmarked rank $r$. Mark this rank and write down its index;
\item
 If we are not in row $1$, then go up one box. If the rank there is $r$ then go to the largest unmarked rank $r$. Mark this rank and write down its index.
\end{enumerate}
\end{enumerate}

\end{algo}

\begin{theo}
  \label{t-I.3}

For a Dyck path $D \in \CD_{\k}$, the permutation $\sigma(D )$ that rearranges the letters of $\SW(D)$ in the order that gives the $SW$ word of $\oD$, is obtained by the Walking Algorithm.
\end{theo}


Let us apply Walking Algorithm (Algorithm \ref{al-Walking Algorithm}) to the index-rank tableau $(T,R)$ in Figure $\ref{fig:Ranking Algorithm}$.
The permutation $\sigma(D )$ is given on the top row. From it one easily produce the middle $\SW(\overline{D})$, whose rank sequence is given in the third row for comparison. Now we clearly see that
the ranks of $\overline{D}$ are exactly the ranks in $R(D)$.
$$
\hskip 0.1in
\left( \begin{array}
  {c}
  \sigma(D) \\ \SW(\overline{D}) \\  r(\overline{D})
\end{array}\right)=
\left(\begin{array}{cccccccccccccccccc}
\!\!\!2&\!\!\!\!\!\!\!\!6 &\!\!\!\!\!\!\!\!4 & \!\!\!\!\!\!\!1 & \!\!\!\!\!\!\!11 & \!\!\!\!\!\!\!8 & \!\!\!\!\!\!\!18 & \!\!\!\!\!\!\!17 & \!\!\!\!\!\!\!15 & \!\!\!\!\!\!13 & \!\!\!\!\!\!10&\!\!\!\!\!\!16 & \!\!\!\!\!\!14& \!\!\!\!\!\!12& \!\!\!\!\!\!9& \!\!\!\!\!\!7& \!\!\!\!\!\!5& \!\!\!\!\!\!3\\
S^{2}&\!\!\!\!\!W&\!\!\!\!\!W&\!\!\!\!\!S^{4}&\!\!\!\!\!W &\!\!\!\!\!S^{5}&\!\!\!\!\! W& \!\!\!\!\!W & \!\!\!\!\!W &\!\!\!\!\!W &\!\!\!\!S^{3}&\!\!\!\!\! W& \!\!\!\!\!W & \!\!\!\!\!W&\!\!\!\!\!W & \!\!\!\!\!W& \!\!\!\!W&\!\!\!\!W\\
\!\!\!0&\!\!\!\!\!\!\!\!2 &\!\!\!\!\!\!\!\!1 & \!\!\!\!\!\!\!0 & \!\!\!\!\!\!\!4 & \!\!\!\!\!\!\!3 & \!\!\!\!\!\!\!8 & \!\!\!\!\!\!\!7 & \!\!\!\!\!\!\!6 & \!\!\!\!\!\!5 & \!\!\!\!\!\!4&\!\!\!\!\!\!7 & \!\!\!\!\!\!6& \!\!\!\!\!\!5& \!\!\!\!\!\!4& \!\!\!\!\!\!3& \!\!\!\!\!\!2& \!\!\!\!\!\!1\\
\end{array}\right).
$$

$$
\hskip 0.15in{
{\oD=\hskip -.4in\vcenter{\includegraphics[width=2.53in]{1.png} } }
\hskip -3.3in
{(T,R)=
\hskip -0.3in
 \vcenter{\includegraphics[width=0.88in]{Tanytableaux.png} }}
\hskip -5.2in
 \vcenter{\includegraphics[width=.88in]{TanytableauxR.png} }}
$$

The paper is organized as follows. In this introduction, we have introduced the main concepts, as well as our main results, especially Theorems \ref{t-I.3} which give linear inverting algorithms
for $\k$-Dyck paths. Section \ref{s-basic} includes the basic facts of the sweep map and the proof of Theorem \ref{t-I.3}, but leave the proof of Lemma \ref{l-length-sigma(D)} in Section \ref{proof}. Lemma \ref{l-length-sigma(D)} says
that our Walking algorithm produces a permutation $\sigma(D)$ of the desired length. To prove this lemma, we give two different approachs in Section \ref{proof}.
\section{Some basic auxiliary facts about the sweep map\label{s-basic}}

There are a number of auxiliary properties of the sweep map in $\k$-Dyck paths that need to be established to prove our basic results.

\begin{lem}
  \label{1-1.2}
The Filling Algorithm terminates only when all indices $1,2,\ldots ,|\k|+n$ have been placed.
\end{lem}

\begin{lem}
  \label{1-1.3}
The Ranking Algorithm \ref{al-Ranking Algorithm} assigned every index a rank.
\end{lem}
\begin{proof}
Assume to the contrary that $i$ with $2\leq i\leq n$ is the smallest such that the top index of the $i$-th column, say $A>1$,
cannot be assigned a rank. This only happens when the index $A-1$ is not assigned a rank yet. But then $A-1$ must not belong to the first $i-1$ columns, contradicting the Filling Algorithm.
\end{proof}

\begin{lem}
  \label{1-1.4}
Let $D\in \CD_{\k}$ be a $\k$-Dyck path with rank tableau $R(D)$. Then the ranks are weakly increasing according to their indices in $T(D)$. In other words,
if indices $1,2, \ldots, n+|\k|$ are assigned ranks $r_1,r_2,r_3,\ldots, r_{n+|\k|}$, then $0=r_1\le r_2 \le r_3 \le \ldots \le r_{|\k|+n}.$
More precisely, $r_{i}-r_{i-1}$ is either $0$ or $1$ for each $2\leq i\leq |\k|+n$.
\end{lem}
\begin{proof}

We prove by induction on $i$.

For the base case $i = 2$, we need to consider the following two cases by using the Filling algorithm \ref{al-Filling Algorithm}.
\begin{enumerate}
\item Index $2$ is in row $1$ column $2$. Then $r_2$ is assigned $0$, so $r_2-r_1=0-0=0$.
\item Index $2$ is placed under index $1$. Then $r_2$ is assigned $ 1$, so $r_2-r_1=1-0=1$.
\end{enumerate}

Now assume by induction that $0=r_1\le r_2 \le r_3 \le \ldots \le r_i$ and that $0\le r_{i}-r_{i-1}\le 1$ for $2 \leq i< |\k|+n$. We need to show that $0\leq r_{i+1}-r_i\le 1$.

There are three cases as follows.

Case 1: If index $i+1$ is in row $1$, then $r_{i+1}=r_i$;

Case 2: If index $i+1$ is placed under index $i$, then $r_{i+1}=r_i+1$;

Case 3: Otherwise, the index $i+1$ is not in row $1$ and is not placed under index $i$. We use the fact that
if the index $j$ with $j\ge 2$ is in row $1$, then $r_j=r_{j-1}$.
Let $i'$ be the smallest index with rank $r_i$. Then $r_{i'}= r_{i'+1}= \ldots = r_i$ and we need to consider the following two cases.

\begin{enumerate}
\item[(i)] If $i'\ge 2$ then it cannot be in the top row, for otherwise $r_{i'-1}=r_{i'}$ contradicting our choice of $i'$. Assume the index above $i'$ is $p$, and the index above $i+1$ is $q$. Then
$p<q<i'<i+1$ by the Filling algorithm,  and we have $r_{i'}=r_p+1,\ r_{i+1} = r_q + 1$ by the Ranking algorithm.
Thus $r_q\leq r_{i'}$ by the induction hypothesis and we obtain $r_{i+1}-r_i\leq r_{i+1}-r_q =1$.
On the other hand, $ r_{i+1}-r_i=r_q-r_p$, which is greater than or equal to $0$ (again) by the induction hypothesis.

\item[(ii)] If $i'=1$ then $0=r_1= r_2  = \cdots = r_i$. It follows that indices $1,2, \ldots, i$ are all in row $1$ and hence $i+1$ is placed under $1$, which implies $0=r_i<r_{i+1}=1$.\qedhere
\end{enumerate}\end{proof}

\begin{lem}\label{l-length-sigma(D)}
  The permutation $\sigma(D)$ produced by Algorithm \ref{al-Walking Algorithm} has length $n+|\k|$.
\end{lem}

We will give two proofs of this lemma in the next section. The first one only considers the equal parameter case $T \in \TAU_n^k$. One will see that the idea extends
but the notation becomes awkward for $T\in \TAU_{\k}$. The second one uses standard terminology from graph theory.

\begin{proof}[Proof of Theorem \ref{t-I.3}]
By Lemma \ref{l-length-sigma(D)}, we may write $\sigma(D)= a_1a_2\cdots a_{|\k|+n}$.
Assume the corresponding ranks are $r_1,r_2,\dots, r_{|\k|+n}$. Define
$\oD$ to be the SW-sequence obtained from $\sigma(D)$ by replacing each top row index $t_i$  with an $S^{k_i}$ and every other index by a $W$.
We need to show that $\Phi(\oD)=D$. This follows from the following two facts.

\begin{enumerate}
  \item[Fact 1:] The rank sequence of $\oD$ is exactly $(r_1,r_2,\dots, r_{|\k|+n})$. This is consistent with the rule $r_{j+1}=r_j+k$ if $a_j$ is equal to $t_i$
and $r_{j+1}=r_j-1$ if $a_j$ is below row $1$.

\item[Fact 2:] The sweep order is from right to left when two ranks are equal. This corresponds to that for equal rank entries, their corresponding indices are increasing from right to left in $\sigma (D)$.
\qedhere
\end{enumerate}

\end{proof}

\section{Proof of Lemma \ref{l-length-sigma(D)}\label{proof}}

\subsection{First proof}
We only consider the equal parameter case $T \in \TAU_n^k$.
We need the following notation. One will see that the notation becomes awkward for $T\in \TAU_\k$.

\def\barwedgex{\wedge}
\def\veebarx{\vee}
Let $D$ be a Dyck path in $\CD_{kn,n}$ with index-rank tableau $(T(D), R(D))$. For any integer $r$, we denote by $n(r)$ the number of $r$'s appearing in $R(D)$,
so $n(r)=0$ if $r<0$. We also denote by $n^-(r)$ and $n^\wedge(r)$ the  number of $r$'s in the top row and the number of $r$'s below the top row respectively.
Similarly, we denote by $n_-(r)$ and $n_\vee(r)$ the number of $r$'s in the bottom row and the number of $r$'s above the bottom row respectively. Then the following three equalities are clear.
\begin{align}\label{e-II.3}
n(r)=n_-(r)+n_\vee(r),\qquad  n_-(r)=n^-(r-k), \qquad n_\vee(r)=n^\wedge(r+1).
\end{align}


\begin{lem}
  \label{1-1.1}
Let $D\in \CD_{m,n}$ be a Dyck path, where $m=kn$. Then we have the following basic properties.
\begin{enumerate}
\item The ranks of $D$ have the common divisor $n$. 

\item For a word $\om\in S^nW^m$ and $1\le i\le m+n$ denote by
$a_i(\om)$ and $b_i(\om)$,  the numbers of  ``$W$'' and ``$S$''
respectively that occur  in the first $i$ letters of $\om$.
It is important to notice that we will have $\om=\SW(D)$ for some
$D\in {\cal D}_{m,n}$  if and only if
 $$ b_i(\om)m-a_i(\om)n\ge 0 \qquad
\hbox{for all $1\le i\le m+n$}.
$$

\item If a rank $r$ appears in $R(D)$, it will appear at most $n$ times and
$$n(r)=n^-(r-k)+n^\barwedgex(r+1).
$$
In particular, we have $$n(0)=n^\barwedgex(1).
$$
\end{enumerate}
\end{lem}

\begin{proof}
\begin{enumerate}
\item Because we start with assigning $0$ to the south end of the first North step, this done we add an $m=kn$ as we go North and subtract an $n$ as we go East, all the ranks are divisible by $n$.

\item In fact after $a_i(\om)$ letters $W$ and $b_i(\om)$ letters $S$, the corresponding path has reached a lattice point of coordinates $\big(a_i(\om),b_i(\om)\big)$, this point is above the diagonal $(0,0)\RA (m,n)$ if and only if
$$
{b_i(\om)\over a_i(\om)}\ge{n\over m}.
$$
\item Since every  column of $R(D)$ is strictly increasing, any rank $r$ may appear at most $n$ times. The equality
$n(r)=n^-(r-k)+n^\barwedgex(r+1)$ follows from equation \eqref{e-II.3}.
\end{enumerate}
\end{proof}

\begin{lem}
  \label{1-1.4m}
In the equal parameter case,
Algorithm \ref{al-Walking Algorithm} terminates only after we mark the smallest rank $1$ entry and write down its index.
\end{lem}
\begin{proof}
Suppose the algorithm terminates after we mark a rank $r$ entry and write down its index $a$.

If $a$ is in row $1$, then all rank $r+k$ entries have been marked. But to mark a rank
$r+k$ entry, we can either go from
a rank $r$ whose index is in row 1, or go from a rank $r+k+1$ entry whose index is not in row 1. Thus the number of marked
rank $r+k$ entry is at most
$n^-(r)+n^\barwedgex(r+k+1)-1$, which is equal to $n(r+k)-1$ by the equality in Lemma \ref{1-1.1}(3). This is a contradiction.

If $a$ is not in row $1$, then $r\geq 1$ and there is no unmarked rank $r-1$ for the walking algorithm to terminate.
We have the following two cases.

\begin{enumerate}
\item If $r>1$, then we show the contradiction that there is at least one unmarked rank $r-1$ entry so that the algorithm will not terminate. Or equivalently the number of
marked $r-1$ is at most $n(r-1)-1$. According to the walking algorithm, to mark a rank $r-1$ entry, we can either go from
a rank $r-1-k$ whose index is in the top row, or go from a rank $r$ entry whose index is not in the top row.
Thus the number of marked rank $r-1$ entry is at most
$n^-(r-1-k)+n^\barwedgex(r)-1$, which is equal to $n(r-1)-1$ by the equality in Lemma \ref{1-1.1}(3).

\item If $r=1$ and we stopped at the non-smallest rank $1$ entry, then we show the contradiction that there is at least one unmarked rank $0$ entry so that the algorithm will not terminate.
Or equivalently the number of marked rank $0$ entries is at most $n(0)-1$. The reason is similar. According to the walking algorithm, to mark a rank $0$ entry, we can either mark the largest rank $0$ entry at the first step,
or go from a rank $1$ entry whose index is not in the top row. Since we stopped at the non-smallest rank $1$ entry, the number of marked rank $0$ entry
is at most
$1+n^\barwedgex(1)-2=n^\barwedgex(1)-1$, which is equal to $n(0)-1$ by the equality in Lemma \ref{1-1.1}(3). \qedhere
\end{enumerate}

\end{proof}

\begin{proof}[First Proof of Lemma \ref{l-length-sigma(D)}]
Now we are ready to show that Algorithm \ref{al-Walking Algorithm} terminates after we write down all of the $m+n$ indices.
Assume to the contrary that the algorithm terminates after we write down only $p<m+n$ indices. Denote the resulting $\sigma(D)$ by $a_{i_1},a_{i_2},a_{i_3},\ldots,a_{i_p}$, with corresponding ranks $r_{i_1},r_{i_2},r_{i_3},\ldots, r_{i_p}$. Then $r_{i_1}=0, r_{i_2}=k$, and $r_{i_p}=1$ is the smallest rank $1$ by Lemma \ref{1-1.4m}. Note that all rank $0$ indices have been written. In particularly, index $1$ is written.
Moreover, for each $j$ with $1\leq j< p$ either $r_{i_j}+k=r_{i_{j+1}}$ or $r_{i_j}-1=r_{i_{j+1}}$.

Arrange the unwritten indices in increasing order as $1<a_{i_{p+1}}< a_{i_{p+2}}<\cdots <a_{i_{m+n}}$, and denote their corresponding ranks by $r_{i_{p+1}}\leq r_{i_{p+2}}\leq r_{i_{p+3}}\leq\ldots\leq r_{i_{m+n}}$.
By our choice $a_{i_{p+1}}-1$ is a written index, so assume $a_{i_j}=a_{i_{p+1}}-1$ for some $j\leq p$. Now we have two cases, both leading to contradictions.

\begin{enumerate}
\item If $r_{i_{p+1}}=r_{i_j}$, then the walking algorithm would have prefered to writting index $a_{i_{p+1}}$ than $a_{i_j}$ when visiting rank $r_{i_j}$.
\item If $r_{i_j}<r_{i_{p+1}}$, then $a_{i_{p+1}}$ is not in row $1$. By definition of $a_{i_{p+1}}$ and the increasing ranks of the indices in Lemma \ref{1-1.4}, all rank $\alpha=r_{i_j}$ entries have been marked. But to mark a rank $\alpha$ entry, we can either go from a rank $\alpha-k$ in top row, or from a rank $\alpha+1$ below top row (including the index $a_{i_{p+1}}$ with $r_{i_{p+1}}$). These implies that
    there are at most $n^{}(\alpha-k)+n^{}(\alpha+1)-1=n(\alpha)-1$ (by \ref{1-1.1}(3)) rank $\alpha$ entries have been marked, a contradiction. \qedhere

\end{enumerate}

\end{proof}

\subsection{Second proof}
Our second proof assumes basic knowledge of graph theory.

Firstly we construct a digraph $G_R$ from the rank tableau $R=R(D)$ as follows. The vertices of $G_R$ are the ranks appearing in $R$. Each directed edge of $G$ is associated to an index of $T$:
i) if the index is $t_i$ and has rank $a$, then the directed edge is $a\to a+k_i$; ii) if the box is not in row $1$ and has rank $b$, then the directed edge
is $b\to b-1$.

Each rank $a$ of $G_R$ is associated with a set $S(a)$ consisting of all indices with rank $a$. Denote by $F(T)=\{t_1,\dots, t_n\}$ the set of fist row indices of $T$ arranged increasingly.

\begin{lem}
The digraph $G_R$ of a rank tableau $R=R(D)$ is balanced. That is, each rank $a$ of $G$ has in-degree equal to out-degree, and equal to $|S(a)|$.
\end{lem}
\begin{proof}
Let $C_i$ be the digraph obtained by restricting $G$ to the $i$-th column of $R$. Then $C_i$ is clearly a directed cycle $r\to r+k_i\to  \cdots \to r+2\to r+1\to r$, where
$r$ is the rank in row $1$. The in-degree and out-degree of  a rank $a$ are both $1$ if $a$ appears in $C_i$ and are $0$ if otherwise.
The lemma then follows since $G$ is the union of $C_i$ for $i=1,2,\dots, n$.
\end{proof}

 The walking algorithm can be restated as a modified Eulerian tour as follows.

\begin{algo}[Walking Algorithm G]
\label{al-Walking Algorithm2}

\noindent
Input: The digraph $G_R$ together with $S(a)$ associated to every vertex $a$, and $F(T)$ as above.

\noindent
Output: A permutation $\sigma(D)$ through walking on $G_R$.

\begin{enumerate}
\item In $G_R$, go to rank $0$, mark the largest index in $S(0)$ and write down it;
\item Repeat the following steps.
\begin{enumerate}
  \item Suppose we are at rank $r$ and has just marked $t_i$ in $F(t)$. If there is no unmarked index in $S(r+k_i)$, then terminates. Otherwise,
   go to rank $r+k_i$, mark the largest unmarked index in $S(r+k_i)$ and  write down it;
\item
 Suppose we are at rank $r$ and has just marked an index not in $F(t)$. If there is no unmarked index in $S(r-1)$, then terminates. Otherwise, go to rank $r-1$, mark the largest unmarked index in $S(r-1)$ and write down it.
\end{enumerate}
\end{enumerate}

\end{algo}

\begin{proof}[Second Proof of Lemma \ref{l-length-sigma(D)}]
Assume to the contrary that the algorithm terminates after we write down $\sigma(D)=a_1a_2\cdots a_p$ for $p<|\k|+n$. Then $r(a_1)=0$.

Firstly we claim that $r(a_p)=1$, and the algorithm terminates when we are trying to go to rank $0$.
This is simply due to the following observations:
i) the in-degree and out-degree of rank $r$ is $|S(r)|$;
ii) every time when marking an index in $S(r)$, we used one in-degree and one out-degree (including the attempt to go out);
iii) the assumption that every index of $S(r)$ has been marked implies that all the in-degree and out-degree has been used,
and hence there is no more edges in $G_R$ directed to rank $r$. The only exceptional case is when $r=0$, because we starting by marking the largest rank $0$ index
without using its in-degree.

Now $C=r(a_1)\to r(a_2) \to \cdots r(a_p) \to r(a_1)$ is a directed cycle contained in $G_R$ as a subgraph.
Let $b$ be the smallest unwritten index. Then $b-1=a_i\ge 1$ for some $i\le p$, since all rank $0$ indices have been written. By Lemma \ref{1-1.4} $r(b)-r(a_i)$ is either $0$ or $1$. Both cases lead to contradictions.
i) If $r(b)= r(a_i)$, then the walking algorithm would have preferred to writing index $b$ than $a_{i}$ when visiting rank $r(b)$.
ii) If $r(b)=r(a_i)+1$, then $b$ is not in row $1$ by the ranking algorithm, and hence will be associated to a directed edge $e=r(b)\to r(b)-1$ in $G_R$.
Now observe that $S(r(b)-1)$ are all contained $\sigma(D)$. This implies that the directed cycle $C$
contains all directed edges into $r(b)-1$, in particular the edge $e$. Then the walking algorithm must have written $b$ already, which contradicts
the choice of $b$.
\end{proof}






\end{document}